\documentclass[
final, nomarks
]{dmtcs-episciences}

\usepackage[utf8]{inputenc}
\usepackage{subfigure}
\usepackage{amsfonts,amsmath,color,amsthm,amssymb}
\usepackage{mathtools}
\usepackage{tikz}
\usepackage{pstricks}
\usepackage{graphicx}
\usepackage{epstopdf}
\usepackage{permute}
\usepackage{latexsym}
\usepackage{mathrsfs}
\usepackage{hyperref}
\usepackage[utf8]{inputenc}
\usepackage{lscape}
\usepackage{longtable}
\usepackage{algorithm2e} 
\usepackage{array}
\usepackage{tikz}

\usepackage{tikz-qtree,tikz-qtree-compat}
\usetikzlibrary{positioning,decorations.pathreplacing}

\author[Godbole and Swickheimer]{Anant Godbole\thanks{Supported by NSF REU Grants 1852171 and 2150434}
  \and Hannah Swickheimer\thanks{Supported by NSF REU Grant 2150434}}

\title[An Alternative Proof for the Expected
Number]{An Alternative Proof for the Expected \\
Number of Distinct Consecutive Patterns in a \\ Random Permutation}
\affiliation{\center
  East Tennessee State University, Johnson City, Tennessee, U.S.A.
}
\newtheorem{thm}{Theorem}[section]
\newtheorem{lem}[thm]{Lemma}

\def\p{\mathbb P}
\def\e{\mathbb E}
\def\l{\lambda}
\def\lr{\left(}
\def\rr{\right)}
\def\lc{\left\{}
\def\rc{\right\}}
\def\P{{\rm Po}}
\def\cl{{\mathcal L}}
\def\n{\noindent}
\def\be{\begin{equation}}
\def\ee{\end{equation}}
\def\tv{d_{TV}}
\newcommand{\beq}{\begin{eqnarray}}
\newcommand{\eeq}{\end{eqnarray}}
\begin{document}
% This is only used if you are compiling for a volume before vol 25
% \publicationdetails{VOL}{2015}{ISS}{NUM}{SUBM}
% This is the new form of collecting the data, starting with vol 25
\publicationdata
{vol. 26:1, Permutation Patterns 2023}{2024}{5}{10.46298/dmtcs.12458}{2023-10-24; 2023-10-24; 2024-03-09; 2024-03-19}{2024-03-19}
%{vol. 25:3 special issue for main purpose}
%{2022}
%{1}
%{10.46298/dmtcs.10472}
%{1998-10-14; 1998-10-14; 2002-07-19; 2014-02-05; 2015-09-09; 2022-12-25}
%{2022-12-3}
%{2022-12-3; None}
%{2023-1-1}
\maketitle
\begin{abstract}
Let $\pi_n$ be a uniformly chosen random permutation on $[n]$.  Using an analysis of the probability that two overlapping consecutive $k$-permutations are order isomorphic, the authors of [2] showed that the expected number of distinct consecutive patterns of all lengths $k\in\{1,2,\ldots,n\}$ in $\pi_n$ is $\frac{n^2}{2}(1-o(1))$ as $n\to\infty$.  This exhibited the fact that random permutations pack consecutive patterns near-perfectly.  We use entirely different methods, namely the Stein-Chen method of Poisson approximation, to reprove and slightly improve their result.
\end{abstract}

%DMTCS is an open access scientific is implemented by the
%\emph{episcience} platform, see \cite{berthaud:hal-01002815} for an
%overview of the strategy. It combines high scientific and editorial
%quality with an open access policy. It is priceless, neither authors
%nor readers pay money for the access. Access is granted by giving
%episcience an irrevocable license to publish the articles, the
%copyright remains with the authors. The platform itself is run by
%French government services that do their best to warrant continuous
%access and a high quality of service.
%
%This document describes the use of the \texttt{dmtcs-episcience.cls}
%document class. It has to be used \emph{for all DMTCS publications}.

\section{Introduction}
 Let $\pi=\pi_n$ be a permutation on $[n]$.  The one-line notation will be used for permutations in this paper; e.g., (2134) is shorthand for $\pi(1)=2; \pi(2)=1; \pi(3)=3; \pi(4)=4.$  We say that $\pi$ {\it contains} a pattern $\mu=\mu_k$ of length $k$ if there are $k$ indices $n_1<n_2<\ldots <n_k$ such that $(\pi(n_1), \pi(n_2),\ldots,\pi(n_k))$ are in the same relative order as $(\mu(1),\mu(2),\ldots, \mu(k))$. We say that $\pi$ {\it consecutively contains} the pattern $\mu_k$  if there are $k$ consecutive indices $(m,m+1,\ldots,m+k-1)$ such that $(\pi(m), \pi(m+1),\ldots,\pi(m+k-1))$ are in the same relative order as $(\mu(1),\mu(2),\ldots, \mu(k))$. Let $\phi(\pi_n)$ be the number of distinct {\it consecutive} patterns of {\it all} lengths $k; 1\le k\le n$, contained in $\pi_n$.  We focus on the case where $\pi$ is a {\it uniformly chosen random permutation} on $[n]$, denote the random value of $\phi(\pi)$ by $X=X_n$, and study, in this paper, its expected value $\e(X)$.  

The authors of \cite{9pp} proposed and used two auxiliary variables, $Y=Y_n$  and $Z=Z_n$ as follows:
$Y=Y_n$ is the number of repeated patterns of any length in $\pi$, and $Y_{n}^{k}=Y^k$ the number of repeated patterns of length $k\in\{1,2,\ldots,n\}$, so that (for a strategically chosen $k_0$),

$$\e(X) \geq \displaystyle\sum_{k=k_0}^{n}(n-k+1) - \e(Y^k).$$  Also we let $Z=Z_n=\sum_k Z^k$ be the number of {\it pairs} of isomorphic patterns, with $Z^k$ being the number of pairs of isomorphic patterns of length $k$.  Generically let $\eta_1$ and $\eta_2$ be two sets of $k$ consecutive positions, and denote by $\rho_1, \rho_2$ the patterns that are present in these two sets of positions.  Thus 
\begin{equation}\e(Z^k) = \displaystyle\sum_{\eta_1}\sum_{\eta_2} \p(\rho_1^k \simeq \rho_2^k),\end{equation}
For example, with $k=7, n=15$, $\eta_1$ and $\eta_2$ might equal $\{1,2,3,4,5,6,7\}$ and $\{5,6,7,8,9,10,11\}$ and $\rho_1^7 \simeq \rho_2^7$ if, e.g., the pattern 7263514 occurs in both sets of positions.
\begin{lem}
    $Y^k \leq Z^k$ for all $k$. 
\end{lem}
\begin{proof}Counting $Y^k$ is equivalent to listing the patterns of length $k$, and starting a tally for any repeats among the $\min\{(n-k+1),k!\}$ patterns that we actually observe.   On the other hand, $Z^k$ increases by one for every pair of the same pattern which means that it also increases more than $Y^k$ if you find more than two of a pattern. 
    For example, if you have three occurrences of the same pattern then $Y^k$ would count the 2 repeats and $Z^k$ would count the ${3\choose 2}$ pairs of repeats. \end{proof}

\bigskip

Since
    $Y^k \leq Z^k$ for all $k$, we have
\begin{equation} \e(X)\geq \sum_{k_0}^n (n-k+1)-\e(Z^k), \end{equation} and, after substantial analysis, to the result in \cite{9pp} that

\begin{eqnarray*}\e(X_n)\ge
\frac{(n-\lceil100\ln n\rceil)^2}{2}&=&\frac{n^2}{2}\lr1-200\frac{\ln n}{n}+\frac{10000\ln^2 n}{n^2}\rr\\
&\ge&\frac{n^2}{2}\lr1-200\frac{\ln n}{n}\rr,
\end{eqnarray*}
for large enough $n$.
In Section 2, we will use the Stein-Chen method of Poisson approximation \cite{bhj} to (slightly) improve the above bound when we prove
\begin{thm} For sufficiently large $n$,
\[\e(X_n)\ge
\frac{n^2}{2}\lr1-\frac{17\ln n}{n}\rr.
\]
\end{thm}
The key difference, besides the use of an entirely different technique, is that fact that we work with the $X$ variable directly, without involving $Y$ and $Z$.  The Stein-Chen method has been applied in several Combinatorics situations in \cite{bhj}; see also, e.g., \cite{nn+1} and \cite{thresh}, where examples are given of the use of the technique in the context of permutations.  Within the domain of Poisson approximation, moreover, we are using it in this paper when the mean of the underlying Poisson distribution is exceptionally small.  This too is unusual.

\noindent {\it Remark on the Non-Consecutive Case}:  In the non-consecutive case, the conjecture is that $\e(X)\sim 2^n$ (rather than $\e(X)\sim \frac{n^2}{2}$), so that once again the expected value of $X$ would be close to its maximum.  When addressing this  conjecture in \cite{6pp}, sole use is made of $X$ when using subadditivity arguments, while the $X,Y,Z$ trifecta is used to get close to proving the conjecture.  Finally the sheer magnitude of the dependencies makes use of Poisson approximation techniques inappropriate in the non-consecutive case.  The same is true of central limit theorems and martingale inequalities.
\section{Poisson Approximation and Consecutive Patterns}
\label{sec:first}

\bigskip

\noindent {\it Proof of Theorem 1.2}:  Since a pattern adds to the tally of distinct patterns if and only if it appears at least once, our key variable $X^k$, the number of distinct patterns of length $k$, can be written as
\[ X^k=\sum_{j=1}^{k!}I({\rm the}\ j^{\rm th}\ {\rm pattern}\ N_j\ {\rm of}\ {\rm length}\ k\ {\rm appears}\  {\rm at}\ {\rm least}\ {\rm once}),\] where $I(A)=1$ iff $A$ occurs ($I(A)=0$ otherwise).  The notation supposes that we have listed the patterns of length $k$ in some fashion, perhaps lexicographically, and label the $j$th pattern as $N_j$.  Thus, 
the expected number of distinct patterns of length $k$ is 
\beq\e(X^k)&=&\sum_{j=1}^{k!}\p({\rm the}\ j^{\rm th}\ {\rm pattern}\ N_j\ {\rm of}\ {\rm length}\ k\ {\rm appears}\ {\rm at}\ {\rm least}\ {\rm once}),\nonumber\\
&=&\sum_{j=1}^{k!}\p(U_{k,j}\ge 1),\eeq
where $U_{k,j}$ is the number of occurrences of the $j$th pattern.  Our analysis will actually bypass the question of {\em which} the $j$th pattern is, and the strategy will be to show that for any $j$,
\be\cl(U_{k,j})\approx\P(\e(U_{k,j})),\ee where for any variable $T$ we denote the distribution of $T$ by $\cl(T)$, and the Poisson variable with parameter $\l$ by $\P(\l)$.    Note that $\e(U_{k,j})=(n-k+1)/k!=\l$ for each $j$.   If (4) were to be shown to be true by proving that
\be
\tv(\cl(U_{k,j}), \P(\l)):=\sup_{A\subseteq {\mathbb Z}^+}\vert\p(U_{k,j}\in A)-\sum_{j\in A}\frac{e^{-\l}\l^j}{j!}\vert\le\varepsilon_{n,k}\to0,
\ee
where $\varepsilon_{n,k}$ does not depend on the pattern, it would follow that for each $j$,
\be \p(U_{k,j}\ge 1)\ge(1-e^{-\l})-\varepsilon_{n,k},\ee
and thus via (3) that 
\be\e(X^k)\ge k!((1-e^{-\l})-\varepsilon_{n,k}).
\ee
We have that 
\be U_{k,j}=\sum_{j=1}^{n-k+1} I_j,\ee
where $I_j$ is the indicator variable that equals 1 if the pattern in question appears in the $k$ places $\{j,j+1,\ldots, j+k-1\}$ starting at $j$.  Also, $I_j$ is independent of the ensemble of $I_\ell$'s whose windows do not intersect those of $I_j$.  Thus Corollary 2.C.5 in \cite{bhj} indicates that
\begin{eqnarray}
&&\tv(\cl(U_{k,j}), \P(\l))\le\frac{1-e^{-\l}}{\l}\cdot\nonumber\\
&&\lr\sum_j\p^2(I_j=1)+\sum_j\sum_{i=j-k+1}^{j+k-1}[\p(I_jI_i=1)+\p(I_j=1)\p(I_i=1])\rr.\nonumber\\
&&
\end{eqnarray}
Since $\p(I_j=1)=\frac{1}{k!}$ for each $j$, we have that $\l=\frac{n-k+1}{k!}$.  Using this fact and bounding $\frac{1-e^{-\l}}{\l}$ by 1, (9) reduces to
\begin{eqnarray}
&&\tv(\cl(U_{k,j}), \P(\l))\nonumber\\ &\le&\frac{(n-k+1)}{k!^2}+\sum_j\sum_{i=j-k+1}^{j+k-1}\lr\p(I_jI_i=1)+\frac{1}{k!^2}\rr\nonumber\\
&\le&\frac{(n-k+1)}{k!^2}+\frac{2(n-k+1)k}{k!^2}+\sum_j\sum_{i=j-k+1}^{j+k-1}\p(I_jI_i=1).\nonumber\\
\end{eqnarray}
Equations (7) and (10) thus give
\begin{eqnarray}
\e(X^k)\nonumber&\ge&k!((1-e^{-\l})-\varepsilon_{n,k})\nonumber\\
&\ge& k!\cdot \bigg((1-e^{-\l})-\frac{(n-k+1)}{k!^2}-\nonumber\\
&&\frac{2(n-k+1)k}{k!^2}-\sum_j\sum_{i=j-k+1}^{j+k-1}\p(I_jI_i=1)\bigg).\end{eqnarray}
We deal separately with the four terms in (11):  

First note that
\begin{eqnarray}
k!(1-e^{-\l})&\ge& k!\frac{\l}{\l+1}\nonumber\\&=&k!\frac{n-k+1}{(\l+1)k!}\nonumber\\&=&\frac{(n-k+1)}{1+\frac{n-k+1}{k!}}\nonumber\\&\ge&(n-k+1)(1-{\frac{n-k+1}{k!}})\nonumber\\
&=&(n-k+1)-\frac{(n-k+1)^2}{k!}.
\end{eqnarray}
We retain the $(n-k+1)$ term in (12) for a later analysis and bound the second term as follows:
\be\frac{(n-k+1)^2}{k!}\le n^2\lr\frac{e}{k}\rr^k,\ee
where we have used the bound $1/k!\le(e/k)^k$.  
Second, in a similar fashion, we have for the second term in (11),
\be\frac{(n-k+1)}{k!}\le n\lr\frac{e}{k}\rr^k\ee
Thirdly, again in a similar fashion, we bound $2(n-k+1)k/k!$ by $2n^2/k!$ and obtain 
\be\frac{2(n-k+1)k}{k!}\le 2n^2\lr\frac{e}{k}\rr^k.\ee  This leaves us with the need to conduct an analysis of 
\[k!\cdot\sum_{j=1}^{n-k+1}\sum_{i=j-k+1}^{j+k-1}\p(I_jI_i=1).\]  In fact, such a correlation analysis is critical to the proof of Theorem 1.2 via the Stein-Chen method.
%and we proceed by recalling the proof of Lemma 2.2, but adapted to fixed consecutive positions for both  the patterns (and furthermore the pattern is fixed too).  Thus the only components relevant are the numbers to be allotted as per the $u_i, l_i$ allocation.  Thus the probability $\p(I_jI_i=1)$ is no more than $2^{2k-2r}/(2k-r)!$ where $r$ is the overlap between the $j$ and $i$ windows.  It follows that for $k\ge 4$
We thus {\it pause the proof of Theorem 1.2} to compute $\p(I_jI_i=1)$ for windows with an overlap of $r$, and the next critical result provides a bound.
\begin{lem} For windows beginning at $i$ and $j$ such that
\[\{j,j+1,\ldots, j+k-1\}\cap\{i,i+1,\ldots, i+k-1\}=r,\] we have
\be\p(I_jI_i=1)\le\frac{2^{2k-2r}}{(2k-r)!}.\ee
\end{lem}
\begin{proof}
Our proof is similar to an argument used in \cite{6pp}.  It is clear that {the patterns in the overlap positions must be isomorphic in order for the pattern in question to exist in both windows}.  For example we could have the pattern being 53412 with $r=3$ (in this case the overlap pattern is 312 for both windows.  The easiest example is for the pattern to be monotone (in this case any $r$ works).  However the pattern cannot be 21534 with $r=2$.  (16) provides a uniform upper bound on $\p(I_jI_i=1)$ that does not pay heed to the fact that $\p(I_jI_i=1)$ might equal zero for some patterns.

To prove (16), we will find a bound on how many ways the numbers between 1 and $2k-r$ can be assigned so that the patterns exist in both windows.  This would yield the numerator in (16).  The denominator of (16), namely $(2k-r)!$, is obvious.

Consider the example below:

\bigskip

    \[ \begin{array}{ccccccccc}
        6 & 5 & 2 & 4 & 3 & 1 & & &  \\
        &&&6&5&2&4&3&1\cr
    \end{array} \]
    \vspace{0.5cm}
\centerline {FIGURE 1:  Consistent Overlap Pattern}

The overlap is a 321 pattern.  The number in the ``1" position has to be 1+2-1=2, since one number in the second occurrence of 652431 has to be assigned and must be lower than it.  Similarly, 3+5-2=6 must be the number associated with the ``2" position.  This is because 5 numbers lower than it still need an assignment.  Finally, the number in the ``3" position must be 4+6-3=7.  In general if the overlap is of magnitude $r$ then the numbers allotted to the $1,2,\ldots,r$ ranks must be, respectively $u_1+l_1-1, u_2+l_2-2,\ldots, u_r+l_r-r$, where the $u_i$'s and $l_i$'s are the upper and lower values of the ranks.  These numbers are determined as indicated.  Now for the rest of the numbers:  

Using the same notation as above, the $u_1+l_1-2$ `low' numbers need to be assigned, $u_1-1$ to the smaller numbers in $\pi_1$ and $l_1-1$ to the smaller numbers in $\pi_2$.  Since it doesn't matter how we do this, there are $ {{u_1+l_1-2}\choose{u_1-1}}$ choices.

The same process is repeated for the positions between the `1' and the `2' in the two patterns.  This can be done in $ {{(u_2+l_2)-(u_1+l_1)-2}\choose{u_2-u_1-1}}$  ways.  In general, when we look at the choices of numbers for the positions between the $``s"$ and the $``s+1"$ in the two patterns, there are 
%\[{{({u_{s+1}+l_{s+1}-(s+1))-(u_s+l_s-s)-1\choose{u_{s+1}-u_s-1}}={{({u_{s+1}+l_{s+1})-(u_s+l_s)-}}\choose{u_{s+1}-u_s-1}}\]  
\[{{(u_{s+1}+l_{s+1}-(s+1))-(u_s+l_s-s)-1)}\choose{u_{s+1}-u_s-1}}\]\[={{u_{s+1}+l_{s+1}-(u_s+l_s)-2}\choose{u_{s+1}-u_s-1}}\]
possibilities.  One choice for the values in the non-overlap positions (for the example in Figure 1) is shown in Figure 2.  The only choice here is for the 3, 4, and 5 numbers; we chose to allot the 3 to the top occurrence and 5 and 4 to the bottom occurrence (in that order).

\[ \begin{array}{ccccccccc}
        6 & 5 & 2 & 4 & 3 & 1 & & &  \\
        &&&6&5&2&4&3&1\cr
{\bf 9}&{\bf 8}&{\bf 3}&7&6&2&{\bf 5}&{\bf 4}&{\bf 1}
    \end{array} \]
    \vspace{0.5cm}
\centerline {FIGURE 2:  Assignment of Numbers to Consistent Overlaps}

\n We next use the crude (but adequate) bound ${a\choose b}\le2^a$ on each term above to get
\[ {{u_1+l_1-2}\choose{u_1-1}}\le 2^{{u_1+l_1-2}};\]
\[{{u_{s+1}+l_{s+1}-(u_s+l_s)-2}\choose{u_{s+1}-u_s-1}}\le 2^{{u_{s+1}+l_{s+1}-(u_s+l_s)-2}}\quad(2\le s\le r)\]
Collapsing the resulting telescoping product, we get the numerator of $2^{2k-2r}$; notice how the ``-2" terms in the exponent collect.  
%\begin{eqnarray*}\p(I_jI_i=1)&\le&{{1}\over {(2k-r)!}}\cdot
%\sup_{u_i, l_i}\binom{u_1+l_1-2}{u_1 -1}\times\nonumber\\
%&&\left(\prod_{i=2}^r \binom{u_i+l_i-u_{i-1}-l_{i-1}-2}{u_i-u_{i-1}-1} \right)\binom{2k -u_r-l_r}{k-u_r}\nonumber\\
%&\le&{2^{2k-2r}}\over{(2k-r)!}
%\end{eqnarray*}
This completes the proof of Lemma 2.1.\end{proof}

\bigskip

\n
Continuing with the proof of Theorem 1.2, the expression below (15) simplifies as 
\begin{eqnarray}k!\cdot\sum_{j=1}^{n-k+1}\sum_{i=j-k+1}^{j+k-1}\p(I_jI_i=1)&\le&2k!\cdot n\sum_{r=1}^{k-1}\p(I_1I_{k+1-r}=1)\nonumber\\
&\le&2k!\cdot n\lc\frac{4}{(k+1)!}+\frac{4^2}{(k+2)!}+\ldots\rc\nonumber\\
&\le&2k!\cdot n\frac{4}{(k+1)!}\lr 1+\frac{4}{(k+1)}+\frac{4^2}{(k+1)^2}\ldots\rc\nonumber\\
&=& \frac{8n}{k-3}.
\end{eqnarray}
Together, Equations (11) through (17) give us the fact that for fixed $k$, 
%\begin{eqnarray}
%\sum_{j=1}^{n-k+1}\sum_{i=j-k+1}^{j+k-1}\p(I_jI_i=1)&\le&\sum_{j=1}^{n-k+1}\sum_{i=j-k+1}^{j+k-1}
%\end{eqnarray}
\be \e(X^k)\ge (n-k+1)-(3+o(1))n^2\lr\frac{e}{k}\rr^k-\frac{8n}{k-3}.
\ee 

Let us consider the last two terms on the right side of (18).  First
\[(3+o(1))n^2\lr\frac{e}{k}\rr^k\le 4n^2\lr\frac{e}{k}\rr^k\le\frac{1}{n}\]
if
\be\log 4+3\log n+k-k\log k\le0,\ee
so that (19) holds if, e.g., $k\ge 4\frac{\log n}{\log\log n}=b_n$.  The second term to evaluate is $8n/(k-3)$ where we'll now assume that $k\ge b_n$.  We get
\begin{eqnarray}
\e(X)&\ge&\sum_{k\ge b_n}\e(X^k)\nonumber\\
&\ge&\sum_{k\ge b_n}\lr(n-k+1)-\frac{1}{n}-\frac{8n}{k-3}\rr\nonumber\\
&\ge&\lr\sum_{k=1}^{n-b_n} k\rr-1-8n\sum_{k\ge b_n}\frac{1}{k-3}\nonumber\\
&\ge&\frac{(n-b_n)^2}{2}-1-8n\log n\nonumber\\
&=&\frac{n^2}{2}\lr1-\frac{b_n}{n}\rr^2-1-8n\log n\nonumber\\
&\ge&\frac{n^2}{2}\lr1-\frac{2b_n}{n}-\frac{2}{n^2}-\frac{16\log n}{n}\rr\nonumber\\
&\ge&\frac{n^2}{2}\lr1-\frac{17\log n}{n}\rr,
\end{eqnarray}
for $n$ large enough, where the last bound in (20) is obtained by writing $b_n =4\frac{\log n}{\log\log n}$.
Equation (20) establishes Theorem 1.2.\hfill\qed
\section{Open Questions}  

\indent\indent(a) Can we gain a deeper insight into the concentration of $X$ around $\e(X)$ by estimating the variance of $X$? 

(b) Might we be able to improve Theorem 1.2 by employing a tighter proof, in particular, by not assuming that all patterns are compatible for any overlap?   Some of the ideas in Borga and Penaguiao \cite{bp1} and \cite{bp2} might help regarding this question.
\section{Acknowledgments}  This research was supported by NSF REU Grants 1852171 and 2150434, and a preliminary version was presented at Permutation Patterns 2022 in Valparaiso, Indiana.  The work has benefited greatly by the suggestions made by the two referees.  One remark, in particular, led to the  uncovering of an error in the previous version.
%\nocite{*}
%\bibliographystyle{plainnat}
%\begin{thebibliography}{99}
%\bibitem{totalAcquisitionRandom} D. Bal, P. Bennett, A. Dudek, and P. Pra$\l$ at (2016+). The total acquisition number of random graphs. arXiv:1402.2854.
%
%\bibitem{integerPartitions} E. Czabarka, M. Marsili, and L. A. Sz\'ekely (2013). Threshold functions for distinct parts: revisiting Erd\H os-Lehner. In {\em Information 
%Theory, Combinatorics, and Search Theory (in Memory of Rudolph 
%Ahlswede)},   eds. H.~Aydinian, F.~Cicalese, C.~Deppe, Springer-Verlag, {\em Lecture Notes in 
%Computer Science}, 7777: 463--471.  
%
%\bibitem{lampertSlater} D.E. Lampert and P. J. Slater. The acquisition number of a graph (1995). {\em Congressus Numerantium}, 109: 203--210.
%
%\bibitem{totalAcquisition} T. LeSaulnier, N. Prince, P. Wenger, D. West, and P. Worah (2013). Total acquisition in graphs. {\em SIAM Journal of Discrete Mathematics}, 27(4):1800--1819.
%
%\bibitem{fractionalAcquisition} P. Wenger (2014). Fractional acquisition in graphs. {\em Discrete Appl. Math}, 178:142--148.
%
%\bibitem {unitAcq}D. West, N. Prince, and P. Wenger. Unit acquisition in graphs. Preprint.
%\end{thebibliography}
\bibliographystyle{plainnat}

\end{document}